\documentclass[a4paper, british, reqno]{amsart}

\usepackage[utf8]{inputenc}
\usepackage[T1]{fontenc}
\usepackage{lmodern}
\usepackage{babel}
\usepackage{csquotes}
\usepackage{enumitem}

\usepackage{amssymb}
\usepackage{calrsfs}
\usepackage{mathabx}
\usepackage{mathtools}
\usepackage{tikz}
\usetikzlibrary{cd, calc}

\usepackage[backend=biber, doi=false, url=false, style=mathalphabetic, maxnames=999, maxalphanames=4, minalphanames=3, giveninits, useprefix, dashed=false, sorting=nyt]{biblatex}
\DeclareLabelalphaNameTemplate{
  \namepart[use=true, pre=true, strwidth=1, compound=true]{prefix}
  \namepart[compound=true]{family}
}
\usepackage[colorlinks, citecolor=blue]{hyperref}
\usepackage[capitalise, noabbrev]{cleveref}
\creflabelformat{equation}{\textnormal{(#2{#1}#3)}}
\Crefname{diagram}{Diagram}{Diagrams}
\creflabelformat{diagram}{\textnormal{(#2{#1}#3)}}

\newcounter{todo}

\newcounter{todonow}

\newcounter{quest}

\newcommand*{\comp}{\circ}
\newcommand*{\conjact}[2]{c^{{#1},{#2}}}
\newcommand*{\coop}[2]{\varphi_{{#1},{#2}}}
\newcommand*{\Cosmash}{\diamond}
\newcommand*{\cosmash}[2]{\iota_{{#1},{#2}}}

\newcommand*{\Def}{\coloneqq}
\newcommand*{\del}{\partial}

\newcommand*{\diag}{\Delta}

\newcommand*{\from}{\colon}

\newcommand*{\join}{\vee}
\newcommand*{\kerb}[2]{\kappa_{{#1},{#2}}}

\DeclarePairedDelimiter{\CoindArr}{\langle}{\rangle}
\DeclarePairedDelimiterX{\Commut}[2]{[}{]}{{#1},{#2}}

\DeclarePairedDelimiterX{\GenRel}[2]{\langle}{\rangle}{{#1} \mathrel{\delimsize|} {#2}}
\DeclarePairedDelimiterX{\SetDelim}[1]{\{}{\}}{
  
  {#1}
}

\DeclareMathOperator{\Img}{Im}

\DeclareMathOperator{\Ker}{Ker}

\newcommand*{\C}{\mathcal{C}}

\newcommand*{\cat}[1]{\textnormal{\textsf{#1}}}

\newcommand*{\Grp}{\cat{Grp}}

\newcommand*{\cond}{\textnormal}
\newcommand*{\PCM}{\cond{(PCM)}}
\newcommand*{\PFF}{\cond{(PFF)}}
\newcommand*{\SH}{\cond{(SH)}}
\newcommand*{\SM}{\cond{(SM)}}

\newcommand*{\mono}{tail}

\newcommand*{\nmono}{Triangle[open, reversed]->}
\newcommand*{\repi}{-Triangle[open]}
\DeclareMathDelimiter{\AMSlrcorner}{\mathclose}{AMSa}{"79}{AMSa}{"79}
\newcommand*{\pb}[2][.2]{
  \arrow[#2, to path={
    let
    \p0 = ($(\tikztotarget.center)-(\tikztostart.center)$)
    in
    (\tikztostart.center) -- +(-#1*\y0,#1*\y0) node{\Large{\(\AMSlrcorner\)}}
  }, phantom]
}
\newcommand*{\po}[2][.2]{
  \arrow[#2, to path={
    let
    \p0 = ($(\tikztotarget.center)-(\tikztostart.center)$)
    in
    (\tikztostart.center) -- +(-#1*\y0,#1*\y0) node{\rotatebox{180}{\Large{\(\AMSlrcorner\)}}}
  }, phantom]
}

\tikzcdset{
  squared/.style={sep={#1em,between origins}},
  squared/.default=4.5
}

\newlist{tfae}{enumerate}{1}
\setlist[tfae]{label=\textnormal{(\roman*)}}
\newlist{assert}{enumerate}{1}
\setlist[assert]{label=\textnormal{(\alph*)}}

\newenvironment*{eqdiagr}
  {\crefalias{equation}{diagram}\begin{equation}}
  {\end{equation}\ignorespacesafterend}

\newtheorem*{theorem*}{Theorem}
\newtheorem{theorem}[subsection]{Theorem}
\newtheorem*{proposition*}{Proposition}
\newtheorem{proposition}[subsection]{Proposition}
\newtheorem*{property*}{Property}

\newtheorem*{properties*}{Properties}

\newtheorem*{lemma*}{Lemma}
\newtheorem{lemma}[subsection]{Lemma}
\newtheorem*{corollary*}{Corollary}
\newtheorem{corollary}[subsection]{Corollary}

\theoremstyle{definition}
\newtheorem*{definition*}{Definition}
\newtheorem{definition}[subsection]{Definition}
\newtheorem*{definitions*}{Definitions}
\newtheorem{definitions}[subsection]{Definitions}
\newtheorem*{notation*}{Notation}

\newtheorem*{notations*}{Notations}

\newtheorem*{example*}{Example}
\newtheorem{example}[subsection]{Example}
\newtheorem*{examples*}{Examples}
\newtheorem{examples}[subsection]{Examples}
\newtheorem*{counterexample*}{Counterexample}

\newtheorem*{counterexamples*}{Counterexamples}

\theoremstyle{remark}
\newtheorem*{remark*}{Remark}
\newtheorem{remark}[subsection]{Remark}
\newtheorem*{remarks*}{Remarks}

\newtheorem*{reminder*}{Reminder}

\newtheorem*{reminders*}{Reminders}

\mathtoolsset{mathic}

\begin{document}

\title[Twisted Commutators and Internal Crossed Modules]{Twisted Commutators and \\ Internal Crossed Modules}
\author{Bo~Shan \textsc{Deval}}
\author{Tim \textsc{Van~der Linden}}

\email[Bo~Shan \textsc{Deval}]{bo.deval@uclouvain.be}
\email[Tim \textsc{Van~der Linden}]{tim.vanderlinden@uclouvain.be}

\address[Bo~Shan \textsc{Deval}, Tim \textsc{Van~der Linden}]{Institut de Recherche en Math\'ematique et Physique, Universit\'e catholique de Louvain, che\-min du cyclotron~2 bte~L7.01.02, B--1348 Louvain-la-Neuve, Belgium}

\thanks{Both authors were funded by the Fonds de la Recherche Scientifique--FNRS\@: the first through a FRIA Doctoral Grant, the second as a Senior Research Associate}

\dedicatory{Dedicated to Hvedri Inassaridze on the occasion of his 90th birthday}

\subjclass[2020]{18D40, 18E13, 18G45}

\keywords{Commutator; internal action; internal crossed module; semi-abelian category}

\begin{abstract}
    We introduce a notion of \emph{relative commutator}---an important special case being \emph{commutators twisted by an action}---as a straightforward modification of the definition of the \emph{Higgins commutator}, establish its relation with a new notion of commutativity---also obtained as a modification of the usual notion---and show how we can use it to characterise \emph{internal crossed modules} in the context of a semi-abelian category.
\end{abstract}

\maketitle

\section{Introduction}
The aim of this article is to initiate a theory of \emph{commutativity twisted by an action}. The wish for such a development arose in a recent investigation in semi-abelian cohomology (on the relation between~\cite{RVdL16} and~\cite{SVdL21}) where in order to better understand a certain higher-dimensional commutator, a decomposition into simpler commutators would be helpful. The situation there does, however, involve non-trivial actions, which calls for an appropriate categorical ``twisted commutator theory''. Our purpose here is not to go into those more involved applications, but rather to develop basic foundations, which we then apply to a simple test case: we characterise internal crossed modules in a semi-abelian category (in the sense of~\cite{JMT02} and~\cite{Jan03}) by means of twisted commutators.

The definition of a \emph{twisted Higgins commutator} is a straightforward modification of the usual definition~\cite{Hig56,MM10b} and agrees with it in the case of a trivial action. In a fixed semi-abelian category, we let an object \(B\) act on an object \(X\) via \(\xi \from B \flat X \to X\). (Details on how such a map codifies an action are recalled in \cref{sec:commutators}; essential for us here is the equivalence between actions and split extensions via a semi-direct product construction.) We consider the induced split short exact sequence (SSES)
\begin{eqdiagr}
    \label{diagr:SES xi}
    \begin{tikzcd}
        0 \arrow[r] & X \arrow[r, "k"] & X \rtimes_\xi B \arrow[r, "d", shift left] & B \arrow[l, "e", shift left] \arrow[r] & 0
    \end{tikzcd}
\end{eqdiagr}
and the following couniversal property: we may ask under which conditions two given morphisms \(f\) and \(g\) as in the diagram
\[
    \begin{tikzcd}[row sep = large]
        X \arrow[r, "k"] \arrow[rd, "f"'] & X \rtimes_\xi B \arrow[d, "\coop{f}{g}" description, dotted] & B \arrow[l, "e"'] \arrow[ld, "g"] \\
        & Y
    \end{tikzcd}
\]
induce a morphism \(\coop{f}{g}\) such that \(\coop{f}{g} \comp k = f\) and \(\coop{f}{g} \comp e = g\). It is clear that this condition---which agrees with the condition that \(f\) and~\(g\) commute in the ordinary sense of~\cite{Huq68,Bou02} when \(\xi\) is a trivial action---may be characterised in terms of equivariance of \(f\) and \(g\) with respect to the action \(\xi\) and the conjugation action of~\(X\) on itself (see, for instance, \cite[Proposition 4.4]{HL13}), but we here follow an alternative approach.

We may mimic the definition of the Higgins commutator as follows (though the Huq commutator would work equally well). We consider the solid part of the diagram below, where the top short exact sequence is induced by taking the kernel of the map \(\langle k,e \rangle\) out of the coproduct of \(X\) and \(B\).
\[
    \begin{tikzcd}[column sep = large]
        0 \arrow[r] & X \Cosmash_\xi B \arrow[r, "\cosmash{k}{e}", \nmono] \arrow[d, \repi, dotted] & X+B \arrow[r, "\CoindArr{k,e}"] \arrow[d, "\CoindArr{f,g}"'] & X \rtimes_\xi B \arrow[ld, "\coop{f}{g}", dashed] \arrow[r] & 0 \\
        & \Commut{f}{g}_\xi \arrow[r, \mono, dotted] & Y
    \end{tikzcd}
\]
This kernel \(X \diamond_\xi B\) is the \emph{twisted cosmash product} of \(X\) and \(B\), induced by \(\xi\); it ``consists of'' twisted formal commutator words in \(X\) and \(B\). The dotted arrows form the image factorisation of the composite \(\CoindArr{f,g} \comp \cosmash{k}{e}\), so that the dashed arrow \(\coop{f}{g}\) exists and satisfies \(\coop{f}{g} \comp k = f\) and \(\coop{f}{g} \comp e = g\) if and only if the thus defined commutator \(\Commut{f}{g}_\xi\) vanishes.

In \cref{sec:commutators}, we investigate this definition in detail, recalling known definitions and results, proving fundamental properties, comparing with related results in the literature. This is then applied in \cref{sec:XMod} where we characterise Janelidze's internal crossed modules~\cite{Jan03} in terms of twisted commutators. Our main result here is \cref{cor:charact cross mod combined} which states that in a semi-abelian category satisfying a certain condition denoted~\SH{}\@, if \(\del \from X \to B\) is an arrow and \(\xi\) an action of \(B\) on~\(X\) inducing the SSES~\eqref{diagr:SES xi}, then the couple \((\del,\xi)\) is an internal crossed module if and only if the twisted commutators \(\Commut{\del}{1_B}_\xi\) and \(\Commut{k}{e \comp \del}_{\conjact{X}{X}}\) vanish.

\section{Relative Commutativity and Commutators}
\label{sec:commutators}
In this section, we work towards the definition of a \emph{commutator twisted by an action}. We start by extending the definition of a commuting pair of arrows to commutativity relative to a chosen cospan and analysing the commutator that results from this. We then focus on the special case of the cospan of monomorphisms in a split extension, which will provide us with the commutator twisted by the action corresponding, through the semi-direct product equivalence, to the given split extension.

\subsection{Relative Commutativity}
In a pointed category with finite products, we say that two coterminal arrows \(f \from X \to Y\) and \(g \from B \to Y\) \emph{(Huq-)commute}~\cite{Huq68} or \emph{cooperate}~\cite[Definition~4.10]{Bou02}
\[
    \begin{tikzcd}[row sep = large]
        X \arrow[r, "{(1_X,0)}"] \arrow[rd, "f"'] & X \times B \arrow[d, "\coop{f}{g}" description, dotted] & B \arrow[l, "{(0,1_B)}"'] \arrow[ld, "g"] \\
        & Y
    \end{tikzcd}
\]
whenever there exists an arrow \({\coop{f}{g} \from X \times B \to Y}\), called a \emph{cooperator of \(f\) and \(g\)}, such that \(\coop{f}{g} \comp (1_X,0) = f\) and \(\coop{f}{g} \comp (0,1_B) = g\).

The goal of this subsection is to generalise this notion of commutativity by replacing the arrows \((1_X,0)\) and \((0,1_B)\) in the definition above by an arbitrary cospan. We will also show that we keep some of the properties of the classical situation. Note that the following notion has already been explored by Martins-Ferreira under the name of \emph{admissible pair} in \cite{Mar10}.

\begin{definition}
    \label{def:rel commutativity}
    Consider a cospan \begin{tikzcd}[cramped, sep=small]
        X \arrow[r, "k"] & A & B \arrow[l, "s"']
    \end{tikzcd} in an arbitrary category. We say that two arrows \({f \from X \to Y}\) and \(g \from B \to Y\) with the same codomain \emph{\((k,s)\)-commute} or \emph{commute relatively to \((k,s)\)} whenever there exists an arrow \(\coop{f}{g} \from A \to Y\) such that \(\coop{f}{g} \comp k = f\) and \(\coop{f}{g} \comp s = g\).
    \[
        \begin{tikzcd}[row sep = large]
            X \arrow[r, "k"] \arrow[rd, "f"'] & A \arrow[d, "\coop{f}{g}" description, dotted] & B \arrow[l, "s"'] \arrow[ld, "g"] \\
            & Y
        \end{tikzcd}
    \]
    This arrow \(\coop{f}{g}\) is called a \emph{\((k,s)\)-cooperator of \(f\) and \(g\)}. When there is no ambiguity, we will drop the \((k,s)\).
\end{definition}

\begin{examples}
    \label{exs:commutativity}
    Here are some examples of commutativity relative to some particular cospans.
    \begin{itemize}
        \item If \(A\) is a zero object, then two arrows commute relatively to the cospan \begin{tikzcd}[cramped, sep=small]
                  X \arrow[r] & 0 & B \arrow[l]
              \end{tikzcd} if and only if both of them are trivial arrows.
              \[
                  \begin{tikzcd}[row sep = large]
                      X \arrow[r] \arrow[rd, "f"'] & 0 \arrow[d, dotted] & B \arrow[l] \arrow[ld, "g"] \\
                      & Y
                  \end{tikzcd}
              \]

        \item Let \(B = A = X\) and \(k = s = 1_X\). Then two arrows \((k,s)\)-commute if and only if they are equal.
              \[
                  \begin{tikzcd}[row sep = large]
                      A \arrow[r, equal] \arrow[rd, "f"'] & A \arrow[d, "\coop{f}{g}" description, dotted] & A \arrow[l, equal] \arrow[ld, "g"] \\
                      & Y
                  \end{tikzcd}
              \]

        \item \emph{Subtraction structures} were introduced in \cite{BJ09} and we will show here that this notion can be expressed as a commutativity condition. More generally, we say that an arrow \(f \from X \to Y\) in a pointed category \emph{admits a subtractor along \(g \from Y \to Z\)}~\cite[Section~1]{Sha23} whenever there is an arrow \(\varphi \from Y \times X \to Z\), called a \emph{subtractor of \(f\) along \(g\)}, making
              \[
                  \begin{tikzcd}[row sep = large]
                      Y \arrow[r, "{(1_Y,0)}"] \arrow[rd, "g"'] & Y \times X \arrow[d, "\varphi" description, dotted] & X \arrow[l, "{(f,1_X)}"'] \arrow[ld, "0"] \\
                      & Z
                  \end{tikzcd}
              \]
              commute---i.e.\ if \(0\) and \(g\) commute relatively to \begin{tikzcd}[cramped]
                  Y \arrow[r, "{(1_Y,0)}"] & Y \times X & X \arrow[l, "{(f,1_X)}"']
              \end{tikzcd} with \(\varphi\) being a cooperator. If \(g\) is the identity on \(Y\), we simply say that \(\varphi\) is a \emph{subtractor of \(f\)}.
    \end{itemize}
\end{examples}

\begin{remark}
    \label{rem:unital}
    In what follows, we will generally work in a finitely complete context and consider the commutativity relative to an \emph{extremally epic} cospan, i.e.\ a cospan \begin{tikzcd}[cramped, sep=small]
        X \arrow[r, "k"] & A & B \arrow[l, "s"']
    \end{tikzcd} such that, if there exists a factorisation
    \[
        \begin{tikzcd}
            & M \arrow[d, "m", \mono] \\
            X \arrow[r, "k"'] \arrow[ru, "k'"] & A & B \arrow[l, "s"] \arrow[lu, "s'"']
        \end{tikzcd}
    \]
    of this cospan through a monomorphism \(m \from M \to A\), then this \(m\) is necessarily an isomorphism. In particular, this implies that this cospan is epic~\cite[Proposition~A.4.3]{BB04} so, in this case, the cooperator is necessarily unique. (For example, this is the case for the classical commutativity in a \emph{unital} context, i.e.\ in a pointed finitely complete category where all cospans of the form \begin{tikzcd}[cramped]
        X \arrow[r, "{(1_X,0)}"] & X \times B & B \arrow[l, "{(0,1_B)}"']
    \end{tikzcd} are extremally epic.)
\end{remark}

Before going to the next subsection where we introduce commutators, let us show some composition and cancellation properties of this relative commutativity (see \cite[Section~1.3]{BB04} for analogous results for classical commutativity).

\begin{proposition}
    \label{prop:postcomp commutativity}
    Consider the following diagram.
    \[
        \begin{tikzcd}[row sep = small]
            X \arrow[r, "k"] \arrow[rd, "f"'] \arrow[rdd, "h \comp f"', bend right] & A & B \arrow[l, "s"'] \arrow[ld, "g"] \arrow[ldd, "h \comp g", bend left] \\
            & Y \arrow[d, "h"] \\
            & Z
        \end{tikzcd}
    \]
    If \(f\) and \(g\) commute relatively to \((k,s)\), then \(h \comp f\) and \(h \comp g\) also commute. Moreover, the converse holds in a finitely complete context if \(h\) is monic and the cospan is extremally epic. The link between their respective cooperators is given by \(\coop{h \comp f}{h \comp g} = h \comp \coop{f}{g}\).
\end{proposition}

\begin{proof}
    The first statement is easily proved by checking the given formula for the cooperators. Thus, let us show the second statement. Assume that \(h\) is monic and \(h \comp f\) and \(h \comp g\) \((k,s)\)-commute. We will show that the cooperator \(\coop{h \comp f}{h \comp g}\) factorises through \(h\) as \(\coop{h \comp f}{h \comp g} = h \comp \varphi\). Since \(h\) is monic, this will imply that \(\varphi\) is a cooperator for \(f\) and \(g\), finishing the proof. Consider the following commutative squares.
    \[
        \begin{tikzcd}
            X \arrow[r, "k"] \arrow[d, "f"'] & A \arrow[d, "\coop{h \comp f}{h \comp g}"] && B \arrow[r, "s"] \arrow[d, "g"'] & A \arrow[d, "\coop{h \comp f}{h \comp g}"] \\
            Y \arrow[r, "h"', \mono] & Z && Y \arrow[r, "h"', \mono] & Z
        \end{tikzcd}
    \]
    By \cite[Proposition~A.4.3]{BB04}, there is a diagonal arrow \(\varphi \from A \to Y\) making the four triangles commute, which is the cooperator of \(f\) and \(g\) as desired.
\end{proof}

\begin{proposition}
    \label{prop:precomp commutativity}
    Consider the following diagram.
    \[
        \begin{tikzcd}
            X' \arrow[r, "k'"] \arrow[d, "x"'] & A' & B' \arrow[l, "s'"'] \arrow[d, "y"] \\
            X \arrow[r, "k"] \arrow[rd, "f"'] & A & B \arrow[l, "s"'] \arrow[ld, "g"] \\
            & Y
        \end{tikzcd}
    \]
    If \(k \comp x\) and \(s \comp y\) commute relatively to \((k',s')\) and \(f\) and \(g\) commute relatively to \((k,s)\), then \(f \comp x\) and \(g \comp y\) commute relatively to  \((k',s')\) with the cooperator \(\coop{f \comp x}{g \comp y} = \coop{f}{g} \comp \coop{k \comp x}{s \comp y}\).
    \[
        \begin{tikzcd}
            X' \arrow[r, "k'"] \arrow[d, "x"'] & A' \arrow[d, "\coop{k \comp x}{s \comp y}" description, dotted] & B' \arrow[l, "s'"'] \arrow[d, "y"] \\
            X \arrow[r, "k"] \arrow[rd, "f"'] & A \arrow[d, "\coop{f}{g}" description, dotted] & B \arrow[l, "s"'] \arrow[ld, "g"] \\
            & Y
        \end{tikzcd}
    \]
\end{proposition}
\begin{proof}
    Again, this follows directly from the formula for the cooperators.
\end{proof}

\begin{remark}
    Let us remark that, in the classical case, the Huq-commutativity condition for \(k \comp x\) and \(s \comp y\) is automatically satisfied, their cooperator being the induced arrow \(x \times y \from X' \times B' \to X \times B\).
\end{remark}

\begin{remark}
    A left-cancellation property analogous to \cite[Proposition~1.6.4]{BB04} might also make sense, but the necessary assumptions are less straightforward and we will not need such a result in this paper.
\end{remark}

\subsection{Relative Commutators}
We will now define a notion of \emph{relative commutator} which, as in the classical case, is a tool that allows us to detect when two arrows commute in the relative sense, or more generally measure ``how close they are to commuting'' when this is not the case. In this subsection, we work in the context of a \emph{normal} category \(\C\), i.e.\ a pointed regular category where every regular epimorphism is a normal epimorphism. We will construct the relative commutator via an image factorisation, so we need the regularity, while  normality is needed to show that a certain regular epimorphism is the cokernel of its kernel, in order to ensure the correspondence between commutators and commutativity (see \cref{prop:rel commutativity}).

As a reminder, in a unital normal category, the \emph{(Higgins) commutator \(\Commut{f}{g}\)} of two arrows \(f \from X \to Y\), \(g \from B \to Y\) with the same codomain~\cites[Definition~5.1]{MM10b}{Hig56} is the image
\[
    \begin{tikzcd}[column sep = large]
        0 \arrow[r] & X \Cosmash B \arrow[r, "\cosmash{X}{B}", \nmono] \arrow[d, \repi, dotted] & X+B \arrow[rr, "\CoindArr{(1_X,0),(0,1_B)}"] \arrow[d, "\CoindArr{f,g}"] && X \times B \\
        & \Commut{f}{g} \arrow[r, \mono, dotted] & Y
    \end{tikzcd}
\]
of the composite \(\CoindArr{f,g} \comp \cosmash{X}{B}\) where \(\cosmash{X}{B}\) is the kernel of the canonical arrow \(\CoindArr{(1_X,0),(0,1_B)} \from X+B \to X \times B\). This commutator was designed to vanish exactly when the arrows \(f\) and \(g\) commute.

We will use a similar procedure to define the commutator relative to a cospan---which we will require to be extremally epic, in order to recover the detection of commutativity; note that it was for the same reason that we gave the definition of the Higgins commutator in a unital context. Similarly to \cref{def:rel commutativity}, the idea is to replace the arrows \((1_X,0)\) and \((0,1_B)\) in the diagram above by an arbitrary extremally epic cospan.

\begin{definition}
    \label{def:rel commutator}
    Let \begin{tikzcd}[cramped, sep=small]
        X \arrow[r, "k"] & A & B \arrow[l, "s"']
    \end{tikzcd} be an extremally epic cospan and \({f \from X \to Y}\), \({g \from B \to Y}\) two arrows with the same codomain. The \emph{\((k,s)\)-commutator \(\Commut{f}{g}_{k,s}\)} of \(f\) and \(g\) is the image of the composite \(\CoindArr{f,g} \comp \cosmash{k}{s}\) where \(\cosmash{k}{s}\) is the kernel of the arrow \(\CoindArr{k,s} \from X+B \to A\). By convention and by coherence with the classical situation, we write \(\cosmash{X}{B}\) for \(\cosmash{(1_X,0)}{(0,1_B)}\). The kernel object of \(\CoindArr{k,s} \from X+B \to A\) is called the \emph{\((k,s)\)-cosmash product} of \(X\) and \(B\) and is denoted \(X \Cosmash_{k,s} B\).
    \[
        \begin{tikzcd}[column sep = large]
            0 \arrow[r] & X \Cosmash_{k,s} B \arrow[r, "\cosmash{k}{s}", \nmono] \arrow[d, \repi, dotted] & X+B \arrow[r, "\CoindArr{k,s}"] \arrow[d, "\CoindArr{f,g}"] & A \\
            & \Commut{f}{g}_{k,s} \arrow[r, \mono, dotted] & Y
        \end{tikzcd}
    \]
\end{definition}

\begin{remark}
    The twisted cosmash product appears in independent recent work of Duvieusart: see~\cite{DuvCT2023}.
\end{remark}

Before we show that these relative commutators detect the relative commutativity, here is a result that tells us that the commutators are preserved by direct images. For an analogous result for the classical commutator, see \cite[Lemma~3.4]{Sha17}. The proof given there is less simple than ours, since it concerns commutators of subobjects, which apparently introduces some technicalities.

\begin{proposition}
    Given the diagram
    \[
        \begin{tikzcd}
            X \arrow[r, "k"] \arrow[rd, "f"'] & A & B \arrow[l, "s"'] \arrow[ld, "g"] \\
            & Y \arrow[d, "h"] \\
            & Z
        \end{tikzcd}
    \]
    with \begin{tikzcd}[cramped, sep=small]
        X \arrow[r, "k"] & A & B \arrow[l, "s"']
    \end{tikzcd} extremally epic, we have the formula
    \[
        \Commut{h \comp f}{h \comp g}_{k,s} = h(\Commut{f}{g}_{k,s})
    \]
    where the latter object is the direct image of \(\Commut{f}{g}_{k,s}\) along \(h\).
\end{proposition}

\begin{proof}
    The formula follows from the diagram
    \[
        \begin{tikzcd}
            X \Cosmash_{k,s} B \arrow[r, \nmono] \arrow[d, \repi, dotted] & X+B \arrow[d, "\CoindArr{f,g}"] \\
            \Commut{f}{g}_{k,s} \arrow[r, \mono, dotted] \arrow[d, \repi, dashed] & Y \arrow[d, "h"] \\
            h(\Commut{f}{g}_{k,s}) \arrow[r, \mono, dashed] & Z
        \end{tikzcd}
    \]
    ---where the dotted and dashed arrows are constructed as the image factorisation of the corresponding squares---by the essential uniqueness of the image factorisation and the fact that, in a regular context, regular epimorphisms compose~\cite[Corollary~A.5.4]{BB04}.
\end{proof}

Let us observe how these commutators detect commutativity relative to an extremally epic cospan.

\begin{proposition}
    \label{prop:rel commutativity}
    Two arrows \(f \from X \to Y\) and \(g \from B \to Y\) commute relatively to an extremally epic cospan \begin{tikzcd}[cramped, sep=small]
        X \arrow[r, "k"] & A & B \arrow[l, "s"']
    \end{tikzcd} if and only if their \((k,s)\)-commutator \(\Commut{f}{g}_{k,s}\) is trivial.
\end{proposition}

\begin{proof}
    First, since \(k\) and \(s\) are jointly extremally epic, \(\CoindArr{k,s}\) is extremally epic so it is a regular epimorphism by regularity~\cite[Proposition~A.4.3, Corollary~A.5.4]{BB04}. Thus, since we are in a normal context, this arrow is a normal epimorphism; in particular, it is the cokernel of its kernel \(\cosmash{k}{s}\). What we need to prove now follows from the universal property of this cokernel \(\CoindArr{k,s}\).
    \[
        \begin{tikzcd}[column sep = large]
            0 \arrow[r] & X \Cosmash_{k,s} B \arrow[r, "\cosmash{k}{s}", \nmono] \arrow[d, \repi, dotted] & X+B \arrow[r, "\CoindArr{k,s}"] \arrow[d, "\CoindArr{f,g}"'] & A \arrow[ld, "\coop{f}{g}", dashed] \arrow[r] & 0 \\
            & \Commut{f}{g}_{k,s} \arrow[r, \mono, dotted] & Y
        \end{tikzcd}
    \]

    If \(f\) and \(g\) commute relatively to \((k,s)\), we have an arrow \(\coop{f}{g}\) such that
    \[
        \coop{f}{g} \comp \CoindArr{k,s} = \CoindArr{\coop{f}{g} \comp k,\coop{f}{g} \comp s} = \CoindArr{f,g}\text{.}
    \]
    Therefore, we have
    \[
        \CoindArr{f,g} \comp \cosmash{k}{s} = (\coop{f}{g} \comp \CoindArr{k,s}) \comp \cosmash{k}{s} = \coop{f}{g} \comp 0 = 0
    \]
    so its image \(\Commut{f}{g}_{k,s}\) is trivial.

    Conversely, if \(\Commut{f}{g}_{k,s} = 0\), then the composite \(\CoindArr{f,g} \comp \cosmash{k}{s}\) is trivial, so we can apply the universal property of \(\CoindArr{k,s}\) to obtain an arrow \(\coop{f}{g} \from A \to Y\) such that \(\coop{f}{g} \comp \CoindArr{k,s} = \CoindArr{f,g}\). Then \(\coop{f}{g} \comp k = f\) and \(\coop{f}{g} \comp s = g\), which finishes the proof.
\end{proof}

\begin{examples}
    Let us examine the commutators associated to some of the cospans of \cref{exs:commutativity}.
    \begin{itemize}
        \item If \(A\) is a zero object, then we can choose \(\cosmash{k}{s} = 1_{X+B}\) so our commutator \(\Commut{f}{g}_{k,s}\) is simply the image \(\Img\CoindArr{f,g}\) of \(\CoindArr{f,g}\).
              \[
                  \begin{tikzcd}[column sep = large]
                      0 \arrow[r] & X+B \arrow[r, equal] \arrow[d, \repi, dotted] & X+B \arrow[r] \arrow[d, "\CoindArr{f,g}"] & 0 \\
                      & \Img\CoindArr{f,g} \arrow[r, \mono, dotted] & Y
                  \end{tikzcd}
              \]

        \item Let \(B = A = X\) and \(k = s = 1_X\). The \((1_X,1_X)\)-cosmash product is called the \emph{difference object \(D(X)\)} of \(X\) and its inclusion into \(X+X\) is denoted \(\delta_X\). We have that \(\Commut{f}{g}_{1_X,1_X}\) is trivial if and only if \(f = g\)---i.e.\ if they \((1_X,1_X)\)-commute, as follows from \cref{prop:rel commutativity}. So this commutator can be seen as a measure of the difference between \(f\) and \(g\).
              \[
                  \begin{tikzcd}[column sep = large]
                      0 \arrow[r] & D(X) \arrow[r, "\delta_X", \nmono] \arrow[d, \repi, dotted] & X+X \arrow[r, "\CoindArr{1_X,1_X}"] \arrow[d, "\CoindArr{f,g}"] & X \\
                      & \Commut{f}{g}_{1_X,1_X} \arrow[r, \mono, dotted] & Y
                  \end{tikzcd}
              \]
    \end{itemize}
\end{examples}

\subsection{Commutators Twisted by an Action}
\label{sec:twisted commut}
A special case of these relative commutators which we will be especially interested in consists of those commutators relative to a cospan induced by an internal action.

In what follows, whenever we consider internal actions, our category \(\C\) will always be semi-abelian in the sense of~\cite{JMT02}. Such a category is indeed normal~\cite[Proposition~3.1.23]{BB04}. Further note that any semi-abelian category is \emph{homological} in the sense of \cite[Definition~4.1.1]{BB04}---see \cite[Proposition~5.1.2]{BB04}---and unital~\cite[Proposition~3.1.18, Proposition~1.8.4]{BB04}---cf.\ \cref{rem:unital}.

Let us recall what is an internal action. An \emph{(internal) action} of an object \(B\) on an object \(X\) is an arrow \(\xi \from B \flat X \to X\) making certain diagrams commute, where \(B \flat X\) is constructed by choosing a coproduct \(B+X\) of \(B\) and \(X\) then taking a kernel \(\kerb{B}{X} \from B \flat X \to B+X\) of the arrow \(\CoindArr{1_B,0} \from B+X \to B\)~\cites[Sections~3.2--3.3]{BJK05}{BJ98}. (This construction gives us a bifunctor \(\flat \from \C \times \C \to \C\).) Such an internal action induces a split short exact sequence (SSES)
\begin{eqdiagr}
    \label{diagr:SES xi p}
    \begin{tikzcd}
        0 \arrow[r] & X \arrow[r, "k"] & X \rtimes_\xi B \arrow[r, "p", shift left] & B \arrow[l, "s", shift left] \arrow[r] & 0
    \end{tikzcd}
\end{eqdiagr}
as in the commutative diagram
\begin{eqdiagr}
    \label{diagr:induced SSES}
    \begin{tikzcd}
        0 \arrow[r] & B \flat X \arrow[r, "\kerb{B}{X}"] \arrow[d, "\xi"'] & B+X \arrow[r, "\CoindArr{1_B,0}", shift left] \arrow[d, "q",  \repi] & B \arrow[l, "\iota_1", shift left] \arrow[d, equal] \\
        0 \arrow[r, dashed] & X \arrow[r, "k = q \comp \iota_2"', dashed] \arrow[ru, "\iota_2"] & X \rtimes_\xi B \arrow[r, "p", dashed, shift left] & B \arrow[l, "s = q \comp \iota_1", dashed, shift left] \arrow[r, dashed] & 0
    \end{tikzcd}
\end{eqdiagr}
where \(\iota_1\), \(\iota_2\) are the inclusions in the coproduct \(B+X\) and \(q\) is constructed as a coequaliser of \(\kerb{B}{X}\) and \(\iota_2 \comp \xi\). This construction is part of an equivalence of categories between internal actions and SSESs~\cite[Section~3]{BJ98}. (In fact, even if \cite{BJ98} discusses an equivalence for each fixed acting object \(B\), these equivalences for every \(B\) induce an equivalence on the level of all internal actions and all SSESs.)

\begin{example}
    \label{ex:triv action}
    An action \(\xi\) whose induced SSES is of the form
    \[
        \begin{tikzcd}
            0 \arrow[r] & X \arrow[r, "{(1_X,0)}"] & X \times B \arrow[r, "\pi_2", shift left] & B \arrow[l, "{(0,1_B)}", shift left] \arrow[r] & 0
        \end{tikzcd}
    \]
    (where \(\pi_2\) is the projection of \(X \times B\) on \(B\)) is said to be a \emph{trivial action}. More explicitly, we have \(\xi = \CoindArr{0,1_X} \comp \kerb{B}{X}\).
\end{example}

\begin{example}
    \label{ex:conj action}
    The \emph{conjugation action} \(\conjact{X}{X}\) of an object \(X\) on itself is defined as the action corresponding to the SSES
    \[
        \begin{tikzcd}
            0 \arrow[r] & X \arrow[r, "{(1_X,0)}"] & X \times X \arrow[r, "\pi_2", shift left] & X \arrow[l, "\diag_X", shift left] \arrow[r] & 0
        \end{tikzcd}
    \]
    where \(\pi_2\) is the projection on the second factor and \(\diag_X \Def (1_X,1_X)\) is the \emph{diagonal} of \(X\). Again, we have an explicit formula \(\conjact{X}{X} = \CoindArr{1_X,1_X} \comp \kerb{X}{X}\).
\end{example}

Let us now define commutators twisted by an action. In order to apply the definition of \((k,s)\)-commutators (\cref{def:rel commutator}), we only need that the cospan \begin{tikzcd}[cramped, sep=small]
    X \arrow[r, "k"] & X \rtimes_\xi B & B \arrow[l, "s"']
\end{tikzcd} is extremally epic, which is indeed the case since \(\C\) is a protomodular category (see \cite[Section~2.4]{JMT02}). Let us remark that for a cospan \((k,s)\) to be induced by an action is a property rather than additional structure. Indeed, whenever it exists, an arrow \(p \from A \to B\) turning this cospan into a SSES is uniquely determined by the identities \(p \comp k = 0\) and \(p \comp s = 1_B\) since \(k\) and \(s\) are jointly epic. If now \(k\) is the kernel of \(p\), then \(p\) is necessarily its cokernel. Note that such an arrow \(p\) is a \((k,s)\)-cooperator of \(0\) and \(1_B\) so the condition that the cospan \begin{tikzcd}[cramped, sep=small]
    X \arrow[r, "k"] & A & B \arrow[l, "s"']
\end{tikzcd} is induced by an action involves the \((k,s)\)-commutativity of \(0\) and~\(1_B\).

\begin{definitions}
    Let \(\xi \from B \flat X \to X\) be an action and~\eqref{diagr:SES xi p} its induced SSES\@. Let \(f \from X \to Y\), \(g \from B \to Y\) be two arrows with the same codomain.
    \begin{itemize}
        \item We say that \(f\) and \(g\) \emph{commute relatively to \(\xi\)} or \emph{\(\xi\)-commute} when they \((k,s)\)-commute. Their \((k,s)\)-cooperator is called the \emph{\(\xi\)-cooperator of \(f\) and \(g\)} (see \cref{def:rel commutativity}).
        \item The \emph{\(\xi\)-commutator of \(f\) and \(g\)} (or \emph{commutator twisted by the action \(\xi\)}) is the \((k,s)\)-commutator of \(f\) and \(g\) (see \cref{def:rel commutator}). We denote it by~\(\Commut{f}{g}_\xi\).
    \end{itemize}
\end{definitions}

\begin{remark}
    \label{rem:triv twisted commut}
    If \(\xi\) is a trivial action, then we regain the definition of the classical commutator: see \cref{ex:triv action}.
\end{remark}

\begin{example}
    Let us recall that an arrow \(f \from X \to Y\) \emph{admits a subtractor \(\varphi\) along \(g \from Y \to Z\)} (see \cref{exs:commutativity}) whenever \(g\) and \(0 \from X \to Z\) commute relatively to \begin{tikzcd}[cramped]
        Y \arrow[r, "{(1_Y,0)}"] & Y \times X & X \arrow[l, "{(f,1_X)}"']
    \end{tikzcd}.
    \[
        \begin{tikzcd}[row sep = large]
            Y \arrow[r, "{(1_Y,0)}"] \arrow[rd, "g"'] & Y \times X \arrow[d, "\varphi" description, dotted] & X \arrow[l, "{(f,1_X)}"'] \arrow[ld, "0"] \\
            & Z
        \end{tikzcd}
    \]
    In fact, we can turn this cospan into the SSES
    \[
        \begin{tikzcd}
            0 \arrow[r] & Y \arrow[r, "{(1_Y,0)}"] & Y \times X \arrow[r, "\pi_2", shift left] & X \arrow[l, "{(f,1_X)}", shift left] \arrow[r] & 0
        \end{tikzcd}
    \]
    so, in a semi-abelian context, this is an example of commutativity twisted by an action. In particular, if \(Y = X\) and \(f\) is the identity on \(X\), then we find the SSES associated with the conjugation action on \(X\): see \cref{ex:conj action}.
\end{example}

We end this section with an interpretation of commutativity twisted by an action in terms of equivariance. In the case of commutators twisted by an action \(\xi\), the relative commutativity of two arrows \(f\) and \(g\) (\cref{def:rel commutativity}) can also be expressed as the equivariance of \(f\) and \(g\) with respect to \(\xi\) and the conjugation action \(\conjact{Y}{Y}\) of~\(Y\) on itself. Note that this corresponds to \cite[Theorem~1.3]{Jan03} (since \(\conjact{Y}{Y} \comp (g \flat f) = \CoindArr{g,f} \comp \kerb{B}{X}\) by definition of \(\conjact{Y}{Y}\)) which provides an alternate proof of this fact. Recall that \(Y \rtimes_{\conjact{Y}{Y}} Y = Y \times Y\): see \cref{ex:conj action}.

\begin{lemma}
    \label{lemma:equivariance => commutativity}
    The arrows \(f \from X \to Y\) and \(g \from B \to Y\) commute relatively to \(\xi\) if and only if they are equivariant with respect to \(\xi\) and the conjugation action \(\conjact{Y}{Y}\) of \(Y\) on itself. The latter condition means that the square
    \[
        \begin{tikzcd}[squared]
            B \flat X \arrow[r, "\xi"] \arrow[d, "g \flat f"'] & X \arrow[d, "f"] \\
            Y \flat Y \arrow[r, "\conjact{Y}{Y}"'] & Y
        \end{tikzcd}
    \]
    commutes or, in terms of SSESs via the equivalence mentioned in the start of the subsection, there exists an arrow \(\psi \from X \rtimes_\xi B \to Y \times Y\) such that \((f,\psi,g)\)
    \[
        \begin{tikzcd}
            0 \arrow[r] & X \arrow[r, "k"] \arrow[d, "f"] & X \rtimes_\xi B \arrow[r, "p", shift left] \arrow[d, "\psi", dotted] & B \arrow[r] \arrow[l, "s", shift left] \arrow[d, "g"] & 0 \\
            0 \arrow[r] & Y \arrow[r, "{(1_Y,0)}"'] & Y \times Y \arrow[r, "\pi_2", shift left] & Y \arrow[l, "\diag_Y", shift left] \arrow[r] & 0
        \end{tikzcd}
    \]
    is a morphism of SSESs.
\end{lemma}
\begin{proof}
    We will show this result from the point of view of SSESs.

    If \((f,g)\) is equivariant, we easily check that \(\coop{f}{g} = \pi_1 \comp \psi\) satisfies the conditions of \cref{def:rel commutativity}.

    Conversely, let us check that \(\psi = (\coop{f}{g}, g \comp p)\) induces a morphism \((f,\psi,g)\) of SSESs, where \(\coop{f}{g}\) is given by the \((k,s)\)-commutativity of \(f\) and \(g\). The equalities \(\psi \comp k = (1_Y,0) \comp f\) and \(g \comp p = \pi_2 \comp \psi\) are immediate and, for the last one, we compute
    \[
        \psi \comp s = \bigl(\coop{f}{g} \comp s, (g \comp p) \comp s\bigr) = (g,g) = \diag_Y \comp g\text{,}
    \]
    which finishes the proof.
\end{proof}

We can also obtain a kind of converse of this result, i.e.\ a commutativity condition to characterise the equivariance of a pair of arrows. For that, we will need a lemma about internal actions, which tells us that, if we restrict the conjugation action of a semi-direct product to its factors, we recover the original action.

\begin{lemma}[{\cite[Corollary~4.5]{HL13}}]
    \label{lemma:conj semi-dir prod}
    Consider an internal action \(\xi \from B \flat X \to X\) and
    \[
        \tag{\ref*{diagr:SES xi}}
        \begin{tikzcd}
            0 \arrow[r] & X \arrow[r, "k"] & X \rtimes_\xi B \arrow[r, "d", shift left] & B \arrow[l, "e", shift left] \arrow[r] & 0
        \end{tikzcd}
    \]
    its induced SSES\@. Then the square
    \[
        \begin{tikzcd}[sep=huge]
            B \flat X \arrow[r, "\xi"] \arrow[d, "e \flat k"'] & X \arrow[d, "k"] \\
            (X \rtimes_\xi B) \flat (X \rtimes_\xi B) \arrow[r, "\conjact{X \rtimes_\xi B}{X \rtimes_\xi B}"'] & X \rtimes_\xi B
        \end{tikzcd}
    \]
    is commutative or, equivalently, there exists \(\psi \from X \rtimes_\xi B \to (X \rtimes_\xi B) \times (X \rtimes_\xi B)\) such that
    \[
        \begin{tikzcd}[sep=large]
            0 \arrow[r] & X \arrow[r, "k"] \arrow[d, "k"] & X \rtimes_\xi B \arrow[r, "d", shift left] \arrow[d, "\psi", dotted] & B \arrow[r] \arrow[l, "e", shift left] \arrow[d, "e"] & 0 \\
            0 \arrow[r] & X \rtimes_\xi B \arrow[r, "{(1_{X \rtimes_\xi B},0)}"'] & (X \rtimes_\xi B) \times (X \rtimes_\xi B) \arrow[r, "\pi_2", shift left] & Y \arrow[l, "\diag_{X \rtimes_\xi B}", shift left] \arrow[r] & 0
        \end{tikzcd}
    \]
    is a morphism of SSESs.
\end{lemma}
\begin{proof}
    Using the formula of \cref{ex:conj action}, we compute
    \[
        \begin{split}
            \conjact{X \rtimes_\xi B}{X \rtimes_\xi B} \comp (e \flat k) &= (\CoindArr{1_{X \rtimes_\xi B},1_{X \rtimes_\xi B}} \comp \kerb{X \rtimes_\xi B}{X \rtimes_\xi B}) \comp (e \flat k) \\
            &= \CoindArr{1_{X \rtimes_\xi B},1_{X \rtimes_\xi B}} \comp \bigl((e+k) \comp \kerb{B}{X}\bigr) \\
            &= \CoindArr{e,k} \comp \kerb{B}{X} = k \comp \xi
        \end{split}
    \]
    as wanted, where the last equality follows from the construction of the induced SSES\@: see \cref{diagr:induced SSES}. (And the second equality follows from the construction of the bifunctor \(\flat\).)
\end{proof}

\begin{corollary}
    \label{coroll:commutativity => equivariance}
    Let \(\xi \from B \flat X \to X\), \(\xi' \from B' \flat X' \to X'\) be two actions and
    \[
        \begin{tikzcd}
            0 \arrow[r] & X \arrow[r, "k"] & X \rtimes_\xi B \arrow[r, "p", shift left] & B \arrow[l, "s", shift left] \arrow[r] & 0\ \text{,} \\
            0 \arrow[r] & X' \arrow[r, "k'"] & X' \rtimes_{\xi'} B' \arrow[r, "p'", shift left] & B' \arrow[l, "s'", shift left] \arrow[r] & 0
        \end{tikzcd}
    \]
    their induced SSESs. Two arrows \(f \from X \to X'\) and \(g \from B \to B'\) are equivariant with respect to \(\xi\) and \(\xi'\) if and only if \(k' \comp f \from X \to X' \rtimes_{\xi'} B'\) and \(s' \comp g \from B \to X' \rtimes_{\xi'} B'\) commute relatively to \(\xi\).
\end{corollary}

\begin{proof}
    By \cref{lemma:equivariance => commutativity}, the commutativity of \(k' \comp f\) and \(s' \comp g\) relative to \(\xi\) is equivalent to the commutativity of the square below.
    \[
        \begin{tikzcd}[sep=huge]
            B \flat X \arrow[r, "\xi"] \arrow[d, "(s' \comp g) \flat (k' \comp f)"'] & X \arrow[d, "k' \comp f"] \\
            (X' \rtimes_{\xi'} B') \flat (X' \rtimes_{\xi'} B') \arrow[r, "\conjact{X' \rtimes_{\xi'} B'}{X' \rtimes_{\xi'} B'}"'] & X' \rtimes_{\xi'} B'
        \end{tikzcd}
    \]
    Now, consider the following diagram.
    \[
        \begin{tikzcd}[sep=huge]
            B \flat X \arrow[r, "\xi"] \arrow[d, "g \flat f"'] & X \arrow[d, "f"] \\
            B' \flat X' \arrow[r, "\xi'"] \arrow[d, "s' \flat k'"'] & X' \arrow[d, "k'"] \\
            (X' \rtimes_{\xi'} B') \flat (X' \rtimes_{\xi'} B') \arrow[r, "\conjact{X' \rtimes_{\xi'} B'}{X' \rtimes_{\xi'} B'}"'] & X' \rtimes_{\xi'} B'
        \end{tikzcd}
    \]
    Since \(k'\) is (split) monic and the bottom square always commutes by \cref{lemma:conj semi-dir prod}, the commutativity of the outer rectangle (which is the square above by functoriality of \(\flat\)) is equivalent to the commutativity of the top square, which correspond to the equivariance of \(f\) and \(g\), finishing the proof.
\end{proof}

\begin{remark}
    This corollary implies that, whenever we consider cospans induced by an action in \cref{prop:precomp commutativity}, we can reformulate the \((k',s')\)-commutativity condition of \(k \comp x\) and \(s \comp y\) as the equivariance of \(x\) and \(y\) with respect to the two actions involved.
\end{remark}

\section{Internal Crossed Modules}
\label{sec:XMod}
The concept of an \emph{internal crossed module} extends the idea behind a crossed module from the category of groups \(\Grp\) to arbitrary semi-abelian categories. This hinges upon the classical equivalence---see for instance~\cite{BS76,MLan98}---between the category of crossed modules (of groups) and the category of internal categories in~\(\Grp\). In~\cite{Jan03}, \emph{internal crossed} and \emph{precrossed modules} were defined as couples \((\del,\xi)\) where \(\del \from X \to B\) is an arrow and \(\xi \from B \flat X \to X\) is an action of \(B\) on~\(X\), satisfying certain compatibility conditions ensuring that the category of internal crossed modules (resp.\ precrossed modules) would be equivalent to the category of internal groupoids (resp.\ reflexive graphs). The aim of this section is to express those conditions in terms of twisted commutators.

\subsection{Definitions}
We make the definitions explicit (rephrasing \textnormal{(2.2)} and \textnormal{(3.15)} in~\cite{Jan03}) by means of the equivalence mentioned in \cref{sec:twisted commut}. Let \(\del \from X \to B\) be an arrow and \(\xi \from B \flat X \to X\) an action with induced SSES
\[
    \tag{\ref*{diagr:SES xi}}
    \begin{tikzcd}
        0 \arrow[r] & X \arrow[r, "k"] & X \rtimes_\xi B \arrow[r, "d", shift left] & B \arrow[l, "e", shift left] \arrow[r] & 0
    \end{tikzcd}\text{.}
\]

\begin{definition}
    \label{def:internal precrossed module}
    A couple \((\del,\xi)\) as above is an \emph{internal precrossed module} whenever there exists an arrow \(c \from X \rtimes_\xi B \to B\) such that \(c \comp e = 1_B\) and \(c \comp k = \del\).
    \[
        \begin{tikzcd}
            0 \arrow[r] & X \arrow[r, "k"] & X \rtimes_\xi B \arrow[r, "d", shift left = 2] \arrow[r, "c"', shift right = 2] & B \arrow[l, "e" description] \arrow[r] & 0
        \end{tikzcd}
    \]
\end{definition}

\begin{remark}
    \label{rem:PXM as RG}
    Since the arrow \(k\) can be recovered as a kernel of \(d\), an internal precrossed module is essentially just a \emph{reflexive graph} \((X \rtimes_\xi B, B, d, c, e)\): a diagram of the form \begin{tikzcd}[cramped]
        C_1 \arrow[r, "d", shift left = 2] \arrow[r, "c"', shift right = 2] & C_0 \arrow[l, "e" description]
    \end{tikzcd} where \(d \comp e = 1_{C_0}\) and \(c \comp e = 1_{C_0}\).
\end{remark}

\begin{definition}
    An internal precrossed module is an \emph{internal crossed module} if, when we see it as a reflexive graph (see \cref{rem:PXM as RG} above), it admits an internal groupoid structure.
\end{definition}

Before proceeding with the characterisations of precrossed and crossed modules in terms of twisted commutators, in the next subsection we first present some definitions and results about reflexive graphs.

\subsection{Reflexive Graphs}
In what follows, unless otherwise stated, we will work in a pointed finitely complete category \(\C\). We recall the definition of the \emph{normalisation} of a reflexive graph as well as some relations between a given reflexive graph and its normalisation.

\begin{definition}
    The \emph{normalisation} of a reflexive graph \begin{tikzcd}[cramped]
        C_1 \arrow[r, "d", shift left = 2] \arrow[r, "c"', shift right = 2] & C_0 \arrow[l, "e" description]
    \end{tikzcd} is the arrow \(c \comp k\) where \(k\) is a kernel of \(d\).
\end{definition}

Originally, we considered the normalisation of effective equivalence relations and, in this case, this normalisation is a normal monomorphism, whence the terminology. It is well known that, in a semi-abelian category, a reflexive graph is an equivalence relation if and only if its normalisation is a (normal) monomorphism. Let us recall the proof:

\begin{lemma}
    \label{lemma:refl rel iff norm monic}
    If a reflexive graph \begin{tikzcd}[cramped]
        C_1 \arrow[r, "d", shift left = 2] \arrow[r, "c"', shift right = 2] & C_0 \arrow[l, "e" description]
    \end{tikzcd} is a \emph{reflexive relation}, i.e.\ if \(d\) and \(c\) are jointly monic, then its normalisation \(c \comp k\), where \(k \from X \to C_1\) is a kernel of \(d\), is monic. Moreover, if we are in a pointed protomodular context, then the converse is also true.
\end{lemma}

\begin{proof}
    Let us first assume that we have a reflexive relation. Then \(d \comp k\) and \(c \comp k\) are jointly monic since \(k\) is a monomorphism. Moreover, \(d \comp k\) is the trivial arrow because \(k\) is a kernel of \(d\) so the previous sentence implies that \(c \comp k\) is monic.

    Now, let us assume that our category \(\C\) is protomodular and that \(c \comp k\) is monic. Consider the following diagram where all the squares except the bottom right one are pullbacks.
    \[
        \begin{tikzcd}[row sep=large]
            X \cap \Ker(c) \pb{rd} \arrow[r] \arrow[d] & \Ker(c) \pb{rd} \arrow[r] \arrow[d, "\ker(c)"'] & 0 \arrow[d] \\
            X \pb{rd} \arrow[r, "k"] \arrow[d] & C_1 \arrow[r, "c"] \arrow[d, "d"'] \arrow[rd, "{(d,c)}" description] & C_0 \\
            0 \arrow[r] & C_0 & C_0 \times C_0 \arrow[u, "\pi_2"'] \arrow[l, "\pi_1"]
        \end{tikzcd}
    \]
    We check that \(X \cap \Ker(c)\) is a kernel of \((d,c)\) (using the universal properties of the two kernels and the pullback involved) and we will show that this object is trivial, which implies that \((d,c)\) is monic (meaning that \(d\) and \(c\) are jointly monic) as needed for \cite[Proposition~3.1.21]{BB04}. Since the two top squares are pullbacks, so is the top rectangle. Therefore, \(X \cap \Ker(c)\) is a kernel of \(c \comp k\). So it is trivial since \(c \comp k\) is monic, which finishes the proof.
\end{proof}

Our second lemma is a characterisation of connected graphs, which is an instance of \cite[Proposition~17]{Bou01}; see also~\cite{RVdL10}.

\begin{lemma}
    \label{lemma:conn refl graph iff norm reg epi}
    In a homological category \(\C\), a reflexive graph \begin{tikzcd}[cramped]
        C_1 \arrow[r, "d", shift left = 2] \arrow[r, "c"', shift right = 2] & C_0 \arrow[l, "e" description]
    \end{tikzcd} is \emph{connected}, i.e.\ the arrow \((d,c) \from C_1 \to C_0 \times C_0\) is a regular epimorphism, if and only if its normalisation \(c \comp k\), where \(k \from X \to C_1\) is a kernel of \(d\), is a regular epimorphism.
\end{lemma}

\begin{proof}
    Let \((d,c)\) be a regular epimorphism. We check that the following square is a pullback since \(k\) is a kernel of \(d\).
    \[
        \begin{tikzcd}[squared=5.5]
            X \pb[.15]{rd} \arrow[r, "k"] \arrow[d, "c \comp k"'] & C_1 \arrow[d, "{(d,c)}"] \\
            C_0 \arrow[r, "{(0,1_{C_0})}"'] & C_0 \times C_0
        \end{tikzcd}
    \]
    Therefore, since \((d,c)\) is a regular epimorphism and, in a regular category, regular epimorphisms are pullback-stable, this diagram tells us that \(c \comp k\) is also a regular epimorphism.

    For the other direction, consider the following diagram where \(k\) and \((0,1_{C_0})\) are kernels of \(d\) and \(\pi_1\), respectively. Here \(d\) is a cokernel of its kernel \(k\), since it is split by \(e\).
    \[
        \begin{tikzcd}
            0 \arrow[r] & X \arrow[r, "k"] \arrow[d, "c \comp k"] & C_1 \arrow[r, "d"] \arrow[d, "{(d,c)}"] & C_0 \arrow[r] \arrow[d, equal] & 0 \\
            0 \arrow[r] & C_0 \arrow[r, "{(0,1_{C_0})}"'] & C_0 \times C_0 \arrow[r, "\pi_1"'] & C_0
        \end{tikzcd}
    \]
    We may check that this diagram commutes. Then, by \cite[Lemma~4.2.5]{BB04}, the morphism \((d,c)\) is a regular epimorphism, since we have a regular epimorphism on the left and an isomorphism on the right.
\end{proof}

Before going to the last definitions of this section, here is an alternative characterisation for connected graphs in an exact Mal'tsev category.

\begin{lemma}
    In an exact Mal'tsev category \(\C\), a reflexive graph \begin{tikzcd}[cramped]
        C_1 \arrow[r, "d", shift left = 2] \arrow[r, "c"', shift right = 2] & C_0 \arrow[l, "e" description]
    \end{tikzcd} is such that the arrow \((d,c) \from C_1 \to C_0 \times C_0\) is an epimorphism if and only if the pushout \(\pi_0 (C)\) of \(d\) and \(c\) is trivial. Moreover, when this is the case, the epimorphism \((d,c)\) is regular so this reflexive graph is connected in the sense of the previous lemma.
\end{lemma}

\begin{proof}
    First, note that the pushout of \(d\) and \(c\) exists by \cite[Theorem~5.7]{CKP93} since \(d\) and \(c\) are split epimorphisms.

    Now, assume that \((d,c)\) is an epimorphism and consider the pushout
    \[
        \begin{tikzcd}[squared]
            C_1 \arrow[r, "c"] \arrow[d, "d"'] & C_0 \arrow[d, "i_2"] \\
            C_0 \arrow[r, "i_1"'] & \pi_0 (C) \po{lu}
        \end{tikzcd}
    \]
    of \(d\) and \(c\). Let \(\pi_1\), \(\pi_2\) be the two projections of the product \(C_0 \times C_0\). We compute
    \[
        (i_1 \comp \pi_1) \comp (d,c) = i_1 \comp d = i_2 \comp c = (i_2 \comp \pi_2) \comp (d,c)
    \]
    so \(i_1 \comp \pi_1\) and \(i_2 \comp \pi_2\) are equal since \((d,c)\) is epic.
    \[
        \begin{tikzcd}[squared=2.25]
            C_0 \times C_0 \arrow[rr, "\pi_2"] \arrow[dd, "\pi_1"'] & & C_0 \arrow[dd] \arrow[rddd, "i_2", bend left] \\
            \\
            C_0 \arrow[rr] \arrow[rrrd, "i_1"', bend right] & & 0 \po{lluu} \arrow[rd, dashed] \\
            & & & \pi_0 (C)
        \end{tikzcd}
    \]
    Therefore, \(i_1\) and \(i_2\) factor through the pushout of \(\pi_1\) and \(\pi_2\), which is trivial since it is the pushout of a product, so \(i_1\) and \(i_2\) are trivial arrows. Finally, since \(i_1\) and \(i_2\) are jointly epic as they form a pushout, this implies that \(\pi_0 (C)\) is trivial---a cospan of trivial arrows is jointly epic only if their codomain is trivial---which finishes the proof of this implication.

    For the other direction, since \(\pi_0 (C)\) is trivial, \(C_0 \times C_0\) is a pullback of \(i_1\) and \(i_2\) so \((d,c)\) is the comparison map between the pushout and this pullback in the diagram below.
    \[
        \begin{tikzcd}[squared=3.5]
            C_1 \arrow[rrrd, "c", bend left] \arrow[rd, "{(d,c)}" description] \arrow[rddd, "d"', bend right] \\
            & C_0 \times C_0 \pb[.15]{rrdd} \arrow[rr, "\pi_2"] \arrow[dd, "\pi_1"'] & & C_0 \arrow[dd, "i_2"] \\
            \\
            & C_0 \arrow[rr, "i_1"'] & & \pi_0 (C) = 0
        \end{tikzcd}
    \]
    Therefore, by \cite[Theorem~5.7]{CKP93} and the fact that \(d\) and \(c\) are regular epimorphisms as they are split, \((d,c)\) is a regular epimorphism as announced.
\end{proof}

Let us now introduce multiplicative and star-multiplicative graphs~\cite[Definitions~4.1 and 5.1]{MM10}, needed for \cref{thm:(SM) = (SH)} below. This theorem is relevant for our characterisation of crossed modules since, in a semi-abelian category, groupoids are exactly multiplicative graphs~\cites[Theorem~2.2]{CPP92}[Proposition~5.1.2]{BB04}.

\begin{definitions}
    \label{def:(star-)mult graph}
    Consider a reflexive graph \begin{tikzcd}[cramped]
        C_1 \arrow[r, "d", shift left = 2] \arrow[r, "c"', shift right = 2] & C_0 \arrow[l, "e" description]
    \end{tikzcd}.
    \begin{itemize}
        \item The reflexive graph \((C_1, C_0, d, c, e)\) is \emph{multiplicative} when there is an arrow \(m \from C_1 \times_{C_0} C_1 \to C_1\) such that \(m \comp (1_{C_1}, e \comp d) = 1_{C_1}\) and \(m \comp (e \comp c, 1_{C_1}) = 1_{C_1}\)
              \[
                  \begin{tikzcd}[sep=huge]
                      C_1 \arrow[r, "{(1_{C_1}, e \comp d)}"] \arrow[rd, equal] & C_1 \times_{C_0} C_1 \arrow[d, "m"] & C_1 \arrow[l, "{(e \comp c, 1_{C_1})}"'] \arrow[ld, equal] \\
                      & C_1
                  \end{tikzcd}
              \]
              where \(C_1 \times_{C_0} C_1\) is constructed as the pullback
              \[
                  \begin{tikzcd}[squared=5.5]
                      C_1 \times_{C_0} C_1 \pb{rd} \arrow[r, "p_1"] \arrow[d, "p_0"'] & C_1 \arrow[d, "c"] \\
                      C_1 \arrow[r, "d"'] & C_0
                  \end{tikzcd}
              \]
              and \((1_{C_1}, e \comp d)\), \((e \comp c, 1_{C_1})\) are the arrows induced by the universal property of this pullback. This arrow \(m\) is called the \emph{multiplication} of \((C_1, C_0, d, c, e)\).
        \item Let \(k \from X \to C_1\) be a kernel of \(d\). The reflexive graph \((C_1, C_0, d, c, e)\) is \emph{star-multiplicative} when there is an arrow \(\zeta \from C_1 \times_{C_0} X \to X\) such that \(\zeta \comp (k,0) = 1_X\) and \(\zeta \comp (e \comp c \comp k, 1_X) = 1_X\)
              \begin{eqdiagr}
                  \label{diagr:star-mult}
                  \begin{tikzcd}[sep=huge]
                      X \arrow[r, "{(k,0)}"] \arrow[rd, equal] & C_1 \times_{C_0} X \arrow[d, "\zeta"] & X \arrow[l, "{(e \comp c \comp k, 1_X)}"'] \arrow[ld, equal] \\
                      & X
                  \end{tikzcd}
              \end{eqdiagr}
              where \(C_1 \times_{C_0} X\) is constructed as the pullback
              \[
                  \begin{tikzcd}[squared=5.5]
                      C_1 \times_{C_0} X \pb{rd} \arrow[r, "\pi_1"] \arrow[d, "\pi_0"'] & X \arrow[d, "c \comp k"] \\
                      C_1 \arrow[r, "d"'] & C_0
                  \end{tikzcd}
              \]
              and \((k,0)\), \((e \comp c \comp k, 1_X)\) are the arrows induced by the universal property of this pullback. This arrow \(\zeta\) is called the \emph{star-multiplication} of \((C_1, C_0, d, c, e)\).
    \end{itemize}
\end{definitions}

\begin{remark}
    \label{rem:mult implies star-mult}
    Every multiplicative graph is star-multiplicative; the idea is that the star-multiplication is a kind of restriction of the multiplication: see \cite{MVdL12} for further details. Consider the arrow \(m \comp \mu\) where \(\mu \from C_1 \times_{C_0} X \to C_1 \times_{C_0} C_1\) is induced in the diagram below by the universal property of the pullback.
    \[
        \begin{tikzcd}[squared=5.5]
            C_1 \times_{C_0} X \arrow[r, "\pi_1"] \arrow[rd, "\mu", dashed] \arrow[d, "\pi_0"'] & X \arrow[rd, "k", bend left] \\
            C_1 \arrow[rd, equal, bend right] & C_1 \times_{C_0} C_1 \pb{rd} \arrow[r, "p_1"] \arrow[d, "p_0"'] & C_1 \arrow[d, "c"] \\
            & C_1 \arrow[r, "d"'] & C_0
        \end{tikzcd}
    \]
    We will show that \(m \comp \mu\) factors through \(X\) and that this factorisation is the wanted star-multiplication. Since \(X\) is a kernel of \(d\), we must show that \(d \comp (m \comp \mu) = 0\). In order to do so, we will precompose this arrow with \((k,0)\) and \((e \comp c \comp k, 1_X)\), which will let us conclude since these two arrows are jointly epic (as \((k,0)\) is a kernel of the splitting \(\pi_1\) of \((e \comp c \comp k, 1_X)\)). By uniqueness, we have \(\mu = (1_{C_1}, e \comp d) \comp \pi_0 = (e \comp c, 1_{C_1}) \comp k \comp \pi_1\) so that
    \[
        d \comp (m \comp \mu) \comp (k,0) = d \comp \pi_0 \comp (k,0) = d \comp k = 0
    \]
    and
    \[
        d \comp (m \comp \mu) \comp (e \comp c \comp k, 1_X) = d \comp k \comp \pi_1 \comp (e \comp c \comp k, 1_X) = d \comp k = 0
    \]
    as expected. Therefore, the universal property of \(k\) gives us an arrow from \(C_1 \times_{C_0} X\) to \(X\) that factorises \(m \comp \mu\) through \(k\) and this is this arrow that we choose to be the star-multiplication \(\zeta\). To conclude, we must show that this \(\zeta\) makes
    \[
        \begin{tikzcd}[sep=huge]
            X \arrow[r, "{(k,0)}"] \arrow[rd, equal] & C_1 \times_{C_0} X \arrow[d, "\zeta"] & X \arrow[l, "{(e \comp c \comp k, 1_X)}"'] \arrow[ld, equal] \\
            & X
        \end{tikzcd}
    \]
    commute and, since \(k\) is monic, we can reduce these two equations by postcomposing them by \(k\). But these two reduced equations correspond exactly to the computations we made just above, the composition with \(d\) intervening only on the last step in both cases, which finishes the proof.
\end{remark}

\begin{remark}
    \label{rem:star-mult commut}
    Given a reflexive graph \((C_1, C_0, d, c, e)\), let us note that the arrow \(\pi_1 \from C_1 \times_{C_0} X \to X\) is at the same time a cokernel of \((k,0) \from X \to C_1 \times_{C_0} X\) and a splitting of \((e \comp c \comp k, 1_X) \from X \to C_1 \times_{C_0} X\). The fact that \(\pi_1 \comp (e \comp c \comp k, 1_X) = 1_X\) is immediate. To show that \(\pi_1\) is a cokernel of \((k,0)\), let us show that \((k,0)\) is kernel of it, the conclusion following since \(\pi_1\) is split. Let \(f \from A \to C_1 \times_{C_0} X\) such that \(\pi_1 \comp f = 0\). Then, by the definition of the pullback, \(f\) is of the form \(f = (g,0)\) with \(g \from A \to C_1\) such that \(d \comp g = (c \comp k) \comp 0 = 0\) and, since \(k\) is a kernel of \(d\), we also have a factorisation \(g = k \comp \varphi\). Finally, we remark that this factorisation \(\varphi \from A \to X\) is the unique arrow such that \((k,0) \comp \varphi = f\), showing that \((k,0)\) is a kernel of \(\pi_1\).
    \[
        \begin{tikzcd}
            A \arrow[rd, "f"] \arrow[d, "\exists! \varphi"', dashed] \\
            X \arrow[r, "{(k,0)}"'] & C_1 \times_{C_0} X \arrow[r, "\pi_1"'] & X
        \end{tikzcd}
    \]
    Therefore, \((k,0)\) and \((e \comp c \comp k, 1_X)\) is the cospan associated to an action \(\xi \from X \to X\) and we can view \cref{diagr:star-mult} as the commutativity of \(1_X\) with itself, twisted by this action \(\xi\). Finally, by \cref{prop:rel commutativity}, this means that the star-multiplicativity of a reflexive graph \((C_1, C_0, d, c, e)\) is characterised by the commutator condition \(\Commut{1_X}{1_X}_\xi = 0\) where \(\xi \from X \to X\) is the internal action associated to the SSES
    \[
        \begin{tikzcd}
            0 \arrow[r] & X \arrow[rr, "{(k,0)}"] && C_1 \times_{C_0} X \arrow[rr, "\pi_1", shift left] && X \arrow[ll, "{(e \comp c \comp k, 1_X)}", shift left] \arrow[r] & 0
        \end{tikzcd}\text{.}
    \]
    (Let us also note that the multiplicativity of a reflexive graph can also be characterised by a commutator condition, but this commutator has no reason to be twisted by an action.)
\end{remark}

\subsection{Characterisation in Terms of Twisted Commutators}
We can now state and prove our characterisations of precrossed and crossed modules in terms of twisted commutators. These are variations of the results of \cite{MM10} so the goal of this section is mainly to reformulate them in terms of twisted commutators via \cref{lemma:equivariance => commutativity} (and provide an alternative proof in the case of \cref{thm:charact cross mod}). In what follows, in a semi-abelian category, we consider an arrow \(\del \from X \to B\) and an action \(\xi \from B \flat X \to X\) with induced SSES~\eqref{diagr:SES xi}.

Let us begin with the characterisation of internal precrossed modules, which is a straightforward consequence of \cref{prop:rel commutativity}.

\begin{theorem}
    \label{thm:charact precross mod}
    In a semi-abelian category, let \(\del \from X \to B\) be an arrow and \(\xi\) an action of \(B\) on \(X\). The couple \((\del,\xi)\) is an internal precrossed module if and only if it satisfies the so-called \emph{Precrossed Module Condition \PCM{}}\@, which is the equality \(\Commut{\del}{1_B}_\xi = 0\).
\end{theorem}

\begin{proof}
    By \cref{prop:rel commutativity}, the triviality of the commutator \(\Commut{\del}{1_B}_\xi\) is equivalent to the existence of an arrow \(\varphi \from X \rtimes_\xi B \to B\)
    \[
        \begin{tikzcd}[row sep = large]
            X \arrow[r, "k"] \arrow[rd, "\del"'] & X \rtimes_\xi B \arrow[d, "\varphi" description, dotted] & B \arrow[l, "e"'] \arrow[ld, "1_B"] \\
            & B
        \end{tikzcd}
    \]
    such that \(\varphi \comp k = \del\) and \(\varphi \comp e = 1_B\), which are exactly the conditions the arrow \(c\) in \cref{def:internal precrossed module} must satisfy.
\end{proof}

The characterisation of internal crossed modules is less simple. In fact, we will even need an additional assumption: we ask that our semi-abelian category \(\C\) satisfies the \emph{Smith-is-Huq Condition \SH{}}~\cite{BG02b}.

\begin{definition}[{\cite[Section~1]{MVdL12}}]
    A semi-abelian category \(\C\) satisfies the \emph{Smith-is-Huq Condition \SH{}} whenever two effective equivalence relations in \(\C\) commute if and only if their normalisations commute.
\end{definition}

Examples include any Moore category in the sense of~\cite{Rod04} and, more generally, any \emph{algebraically coherent} semi-abelian category~\cite{CGVdL15}. We will not use this definition directly but the following characterisation instead.

\begin{theorem}[{\cite[Theorem~3.8]{MVdL12}}]
    \label{thm:(SM) = (SH)}
    For a semi-abelian category, the following conditions are equivalent:
    \begin{tfae}
        \item[\SM{}] every star-multiplicative graph is multiplicative,
        \item[\SH{}] two (effective) equivalence relations commute if and only if their normalisations commute.\qed
    \end{tfae}
\end{theorem}

Indeed, since an internal crossed module is an internal precrossed module that is also a groupoid and, as said earlier, in a semi-abelian context, multiplicative graphs coincide with groupoids, we see that the condition \SM{} might be useful for our desired characterisation. We now have all the necessary tools for its proof.

\begin{theorem}
    \label{thm:charact cross mod}
    In a semi-abelian category, let the couple \((\del,\xi)\) be an internal precrossed module with~\eqref{diagr:SES xi} the SSES associated to \(\xi\). Then \((\del,\xi)\) satisfies the so-called \emph{Peiffer Condition \PFF{}}\@, which is the equality \(\Commut{k}{e \comp \del}_{\conjact{X}{X}} = 0\), if and only if the reflexive graph \((X \rtimes_\xi B, B, d, c, e)\) is star-multiplicative.

    In particular, by \cref{rem:mult implies star-mult} and \cref{thm:(SM) = (SH)}, this implies that, if \((\del,\xi)\) is an internal crossed module, then \PFF{} is satisfied and, conversely, if the category \(\C\) satisfies \SH{}\@, then the condition \PFF{} implies that \((\del,\xi)\) is an internal crossed module.
\end{theorem}

\begin{proof}
    Let us recall that
    \[
        \begin{tikzcd}
            0 \arrow[r] & X \arrow[r, "{(1_X,0)}"] & X \times X \arrow[r, "\pi_2", shift left] & X \arrow[l, "\diag_X", shift left] \arrow[r] & 0
        \end{tikzcd}\text{,}
    \]
    where \(\pi_2\) is the projection on the second factor and \(\diag_X\) is the diagonal of \(X\), is the SSES associated to \(\conjact{X}{X}\).

    Let us start with \(\Commut{k}{e \comp \del}_{\conjact{X}{X}} = 0\). By \cref{prop:rel commutativity}, this gives us an arrow \(\varphi \from X \times X \to X \rtimes_\xi B\) such that
    \[
        \begin{tikzcd}
            X \arrow[r, "{(1_X,0)}"] \arrow[rd, "k"'] & X \times X \arrow[d, "\varphi" description, dotted] & X \arrow[l, "\diag_X"'] \arrow[ld, "e \comp \del"] \\
            & X \rtimes_\xi B
        \end{tikzcd}
    \]
    commutes. Then this arrow \(\varphi\) makes
    \[
        \begin{tikzcd}
            0 \arrow[r] & X \arrow[r, "{(1_X,0)}"] \arrow[d, equal] & X \times X \arrow[r, "\pi_2", shift left] \arrow[d, "\varphi", dotted] & X \arrow[r] \arrow[d, "\del"] \arrow[l, "\diag_X", shift left] & 0 \\
            0 \arrow[r] & X \arrow[r, "k"] & X \rtimes_\xi B \arrow[r, "d", shift left] & B \arrow[r] \arrow[l, "e", shift left] & 0
        \end{tikzcd}
    \]
    commute. Indeed, we must just check that \(d \comp \varphi = \del \comp \pi_2\) and we see that this is the case by precomposing with \((1_X,0)\) and \(\diag_X\) which are jointly epic. Therefore, the right-hand square whose the horizontal arrows point to the right is a pullback, since we have an isomorphism on the left~\cite[Lemma~4.2.5]{BB04} so this square is the pullback mentioned in the definition of a star-multiplicative graph (see \cref{def:(star-)mult graph}). (Note that, by definition of precrossed module, we do indeed have \(\del = c \comp k\).) Finally, we check that \(\pi_1 \from X \times X \to X\) is the needed star-multiplication: we have the commutative diagram
    \[
        \begin{tikzcd}
            X \arrow[r, "{(1_X,0)}"] \arrow[rd, equal] & X \times X \arrow[d, "\pi_1"] & X \arrow[l, "\diag_X"'] \arrow[ld, equal] \\
            & X
        \end{tikzcd}
    \]
    and we easily check that \((1_X,0)\) and \(\diag_X\) are the required arrows (induced by the universal property of the pullback) in \cref{diagr:star-mult}.

    Conversely, if we have a star-multiplication \(\zeta \from (X \rtimes_\xi B) \times_B X \to X\) on the reflexive graph associated to our precrossed module, with \((X \rtimes_\xi B) \times_B X \to X\) being constructed as the pullback
    \[
        \begin{tikzcd}[squared=5.5]
            (X \rtimes_\xi B) \times_B X \pb{rd} \arrow[r, "p_1"] \arrow[d, "p_0"'] & X \arrow[d, "c \comp k = \del"] \\
            X \rtimes_\xi B \arrow[r, "d"'] & C_0
        \end{tikzcd}\text{,}
    \]
    then, by \cref{prop:rel commutativity}, we must show that \(\Commut{k}{e \comp \del}_{\conjact{X}{X}} = 0\), i.e.\ that there is an arrow \(\varphi \from X \times X \to X \rtimes_\xi B\) that makes
    \begin{eqdiagr}
        \label{diagr:wanted varphi}
        \begin{tikzcd}
            X \arrow[r, "{(1_X,0)}"] \arrow[rd, "k"'] & X \times X \arrow[d, "\varphi" description, dotted] & X \arrow[l, "\diag_X"'] \arrow[ld, "e \comp \del"] \\
            & X \rtimes_\xi B
        \end{tikzcd}
    \end{eqdiagr}
    commute. As a reminder, the star-multiplication \(\zeta\) is defined as the arrow making the following diagram commute.
    \[
        \begin{tikzcd}[sep=large]
            X \arrow[r, "{(k,0)}"] \arrow[rd, equal] & (X \rtimes_\xi B) \times_B X \arrow[d, "\zeta"] & X \arrow[l, "{(e \comp \del, 1_X)}"'] \arrow[ld, equal] \\
            & X
        \end{tikzcd}
    \]
    The first implication of this proposition tells us that, if we have a crossed module, then the product of \(X\) with itself coincides (up to isomorphism) with the pullback \((X \rtimes_\xi B) \times_B X\) (with different projections) and that the wanted arrow \(\varphi\) is the first projection of this pullback, so let us show that it is also the case here. Consider the following diagram.
    \[
        \begin{tikzcd}
            X \arrow[r, "{(k,0)}"] \arrow[d, equal] & (X \rtimes_\xi B) \times_B X \arrow[r, "p_1", shift left] \arrow[d, "{(\zeta,p_1)}"] & X \arrow[l, "{(e \comp \del,1_X)}", shift left] \arrow[d, equal] \\
            X \arrow[r, "{(1_X,0)}"'] & X \times X \arrow[r, "\pi_2", shift left] & X \arrow[l, "\diag_X", shift left]
        \end{tikzcd}
    \]
    We check that this diagram commutes and that both rows are (split) exact sequences. The arrow \((k,0)\) is indeed a kernel of \(p_1\) since \(k\) is the kernel of \(d\) and \(p_1\) is the pullback of~\(d\). Therefore, we can apply the (Split) Short Five Lemma~\cite[Theorem~4.1.10]{BB04} to deduce that \((\zeta,p_1)\) is an isomorphism. Finally, we see that \(\varphi = p_0 \comp (\zeta,p_1)^{-1}\) is the wanted arrow making \cref{diagr:wanted varphi} commute, which finishes the proof.
\end{proof}

\begin{remark}
    \label{rem:Peiffer graph}
    Let us justify our terminology by remarking that, by \cref{coroll:commutativity => equivariance} and \cref{prop:rel commutativity}, the Peiffer Condition \(\Commut{k}{e \comp \del}_{\conjact{X}{X}} = 0\) is equivalent to the commutativity of the square
    \[
        \begin{tikzcd}[squared]
            X \flat X \arrow[r, "\conjact{X}{X}"] \arrow[d, "\del \flat 1_X"'] & X \arrow[d, equal] \\
            B \flat X \arrow[r, "\xi"'] & X
        \end{tikzcd}
    \]
    which (via the equivalence between actions and SSESs) corresponds to the definition of Peiffer graph of \cite[Definition~5.2]{MM10}. Moreover, it was already shown in \cite[Theorem~5.3]{MM10} that Peiffer graphs and star-multiplicative graphs coincide, so this gives an alternative proof for \cref{thm:charact cross mod}.
\end{remark}

These two theorems combine to:
\begin{corollary}
    \label{cor:charact cross mod combined}
    In a semi-abelian category satisfying \SH{}\@, let \(\del \from X \to B\) be an arrow and \(\xi\) an action of \(B\) on \(X\) inducing the SSES~\eqref{diagr:SES xi}. The couple \((\del,\xi)\) is an internal crossed module if and only if the twisted commutators \(\Commut{\del}{1_B}_\xi\) and \(\Commut{k}{e \comp \del}_{\conjact{X}{X}}\) vanish.\qed
\end{corollary}

\subsection{Two Special Cases of Crossed Modules}
Let us see what this characterisation amounts to in two well-known special types of crossed modules: those \((\del,\xi)\) where \(\del\) is a monomorphism and the ones where \(\del\) is a regular epimorphism.

\subsubsection*{Monic \(\del\)}
We start with the case where \(\del\) is monic. Then \((\del,\xi)\) is a crossed module if and only if \(\del \from X \to B\) is the inclusion of a normal subobject \(X\) of \(B\) and \({\xi \from B \flat X \to X}\) is the action of conjugation of \(B\) on \(X\), i.e.\ the unique action making
\begin{eqdiagr}
    \label{diagr:conj action}
    \begin{tikzcd}[squared]
        B \flat X \arrow[r, "\xi"] \arrow[d, "B \flat \del"'] & X \arrow[d, "\del"] \\
        B \flat B \arrow[r, "\conjact{B}{B}"'] & B
    \end{tikzcd}
\end{eqdiagr}
commute~\cite[Proposition~7.2]{JMU07}. Note that the uniqueness follows from the fact that \(\del\) is monic. The given reference states the equivalence for precrossed modules but, as we will see below, any precrossed module \((\del,\xi)\) with \(\del\) monic is automatically a crossed module.

Let us first assume that \((\del,\xi)\) is just a precrossed module. By \cref{thm:charact precross mod}, this is equivalent to the triviality of the commutator \(\Commut{\del}{1_B}_\xi\). Moreover, by \cref{prop:rel commutativity} and \cref{lemma:equivariance => commutativity}, we can reformulate this condition as the equivariance of~\(\del\) and \(1_B\) with respect to \(\xi\), which means exactly that \(\xi\) is the conjugation action of \(B\) on \(X\) (see \eqref{diagr:conj action}). Finally, the existence of such a conjugation action is equivalent to the normality of the subobject \(\del \from X \to B\) (\cite[Corollary~2.3]{JMU09} combined with \cite[Proposition~p.~382]{JMT02}---as a reminder, semi-abelian categories are normal so we can indeed apply \cite[Proposition~p.~382]{JMT02} on the square of \cite[Corollary~2.3.(b)]{JMU09}), which is what was announced in the previous paragraph.

Thus we already know that a couple \((\del,\xi)\) with \(\del\) monic is a precrossed module if and only if \(\del\) is the inclusion of a normal subobject. Now, let us see what tells us the Peiffer Condition, i.e.\ the triviality of \(\Commut{k}{e \comp \del}_{\conjact{X}{X}}\) (see \cref{thm:charact cross mod}). Using again \cref{prop:rel commutativity} and \cref{lemma:equivariance => commutativity}, this condition is equivalent to the commutativity of the diagram
\[
    \begin{tikzcd}[sep=huge]
        X \flat X \arrow[r, "\conjact{X}{X}"] \arrow[d, "(e \comp \del) \flat k"'] & X \arrow[d, "k"] \\
        (X \rtimes_\xi B) \flat (X \rtimes_\xi B) \arrow[r, "\conjact{X \rtimes_\xi B}{X \rtimes_\xi B}"'] & X \rtimes_\xi B
    \end{tikzcd}
\]
which can be decomposed as
\[
    \begin{tikzcd}[sep=huge]
        X \flat X \arrow[r, "\conjact{X}{X}"] \arrow[d, "\del \flat 1_X"'] & X \arrow[d, equal] \\
        B \flat X \arrow[r, "\xi"] \arrow[d, "e \flat k"'] & X \arrow[d, "k"] \\
        (X \rtimes_\xi B) \flat (X \rtimes_\xi B) \arrow[r, "\conjact{X \rtimes_\xi B}{X \rtimes_\xi B}"'] & X \rtimes_\xi B
    \end{tikzcd}\text{.}
\]
Since we already showed that \(\xi\) is a conjugation action, the commutativity of the top square is quite expected. More formally, let us show that these two arrows from \(X \flat X\) to \(X\) are equal by postcomposing them with \(\del\), which is monic by assumption. Using the fact that \(\xi\) is the conjugation action of \(B\) on \(X\) as well as the explicit formula for conjugation actions from \cref{ex:conj action}, we compute
\[
    \begin{split}
        \del \comp \bigl(\xi \comp (\del \flat 1_X)\bigr) &= \bigl(\conjact{B}{B} \comp (B \flat \del)\bigr) \comp (\del \flat 1_X) = \conjact{B}{B} \comp (\del \flat \del) \\
        &= (\CoindArr{1_B,1_B} \comp \kerb{B}{B}) \comp (\del \flat \del) = \CoindArr{1_B,1_B} \comp \bigl((\del+\del) \comp \kerb{X}{X}\bigr) \\
        &= \CoindArr{\del,\del} \comp \kerb{X}{X} = \del \comp (\CoindArr{1_X,1_X} \comp \kerb{X}{X}) = \del \comp \conjact{X}{X}
    \end{split}
\]
as wanted. Finally, the bottom square commutes by \cref{lemma:conj semi-dir prod}. In conclusion, we showed that, if \(\del\) is a monomorphism, precrossed modules and crossed modules coincide and correspond to the normal subobjects. (We need the condition \SH{} in order to apply \cref{thm:charact cross mod} for the crossed modules but this characterisation of normal subobjects remains true in an arbitrary semi-abelian category. Indeed, if we have a precrossed module with \(\del \from X \to B\) monic, the corresponding reflexive graph is a reflexive relation by \cref{lemma:refl rel iff norm monic} so it is an equivalence relation since our category is Mal'tsev~\cite[Proposition~5.1.2]{BB04}. Therefore, it is transitive and this gives us a multiplication, making our precrossed module an internal groupoid, hence a crossed module.)

\subsubsection*{Regularly epic \(\del\)}
For our second example, let us assume that \(\del\) is a regular epimorphism. In this case, it is known~\cite[Corollary~3.1]{BG02} (combined with \cite[Remark~3.1]{BG02}, \cite[Proposition~2.2]{GVdL08} and \cref{lemma:conn refl graph iff norm reg epi}) that these crossed modules correspond to \emph{central extensions}. Such is a regular epimorphism \(\del \from X \to B\) whose kernel \(\kappa \from K \to X\) Huq-commutes with its domain \(1_X\) or, in other words, for which there exists an arrow \(\psi \from K \times X \to X\) making the diagram below commute.
\[
    \begin{tikzcd}[sep = large]
        K \arrow[r, "{(1_K,0)}"] \arrow[rd, "\kappa"'] & K \times X \arrow[d, "\psi" description, dotted] & X \arrow[l, "{(0,1_X)}"'] \arrow[ld, equal] \\
        & X
    \end{tikzcd}
\]
This time, we will need \SH{} for our proof---note, however, that this result nevertheless holds in an arbitrary semi-abelian category, see the references mentioned above.

Let us now examine how the characterisation mentioned above express itself in terms of commutator conditions. First, the fact that a crossed module \((\del,\xi)\) with \(\del\) being a regular epimorphism is a central extension follows directly from the following lemma.

\begin{lemma}
    \label{lemma:kappa central}
    For a crossed module \((\del,\xi)\), the kernel \(\kappa \from K \to X\) and the domain \(1_X\) of \(\del\) commute (i.e.\ \(\kappa\) is central).
\end{lemma}

\begin{proof}
    Let us consider the following commutative diagram
    \[
        \begin{tikzcd}[sep=large]
            K \flat X \arrow[rr, "\zeta"] \arrow[d, "\kappa \flat 1_X"'] & & X \arrow[d, equal] \\
            X \flat X \arrow[rr, "\conjact{X}{X}"] \arrow[d, "(e \comp \del) \flat k"'] & & X \arrow[d, "k"] \\
            (X \rtimes_\xi B) \flat (X \rtimes_\xi B) \arrow[rr, "\conjact{X \rtimes_\xi B}{X \rtimes_\xi B}"'] & & (X \rtimes_\xi B)
        \end{tikzcd}
    \]
    where \(\zeta \Def \conjact{X}{X} \comp (\kappa \flat 1_X)\) is the induced action of \(K\) on \(X\) and the bottom square commutes by \cref{thm:charact cross mod} combined with \cref{lemma:equivariance => commutativity}. We will show that \(\zeta\) is the trivial action of \(K\) on \(X\), which will finish the proof since the commutativity of the top square then  means that \(\kappa\) and \(1_X\) are equivariant with respect to the trivial action \(\zeta\) and the conjugation action \(\conjact{X}{X}\), i.e.\ that \(\kappa\) and \(1_X\) commute by \cref{lemma:equivariance => commutativity} (remember that Huq-commutativity corresponds to commutativity relative to the trivial action: see \cref{rem:triv twisted commut}).

    We compute
    \[
        \begin{split}
            k \comp \zeta &= \conjact{(X \rtimes_\xi B)}{(X \rtimes_\xi B)} \comp \Bigl(\bigl((e \comp \del) \flat k\bigr) \comp (\kappa \flat 1_X)\Bigr) \\
            &= (\CoindArr{1_{X \rtimes_\xi B},1_{X \rtimes_\xi B}} \comp \kerb{X \rtimes_\xi B}{X \rtimes_\xi B}) \comp (0 \flat k) \\
            &= \CoindArr{1_{X \rtimes_\xi B},1_{X \rtimes_\xi B}} \comp \bigl((0+k) \comp \kerb{K}{X}\bigr) = \CoindArr{0,k} \comp \kerb{K}{X} \\
            &= k \comp (\CoindArr{0,1_X} \comp \kerb{K}{X})\text{,}
        \end{split}
    \]
    where we used that \(\kappa\) is a kernel of \(\del\) for the second equality, so \(\zeta = \CoindArr{0,1_X} \comp \kerb{K}{X}\) since \(k\) is monic, which is indeed the trivial action of \(K\) on \(X\) as announced (see \cref{ex:triv action}).
\end{proof}

\begin{remark}
    Let us present an alternative proof for this lemma using \cite[Theorem~4.1]{Sha23} which gives a characterisation of central arrows (note that a semi-abelian category is indeed regular---by definition---and subtractive---see \cite[Proposition~3]{Jan05} and \cite[Proposition~3.1.18]{BB04}). Consider the following commutative diagram
    \[
        \begin{tikzcd}
            X \arrow[r, "{(1_X,0)}"] \arrow[d, equal] & X \times K \arrow[d, "1_X \times \kappa"] & K \arrow[l, "{(\kappa,1_K)}"'] \arrow[d, "\kappa"] \\
            X \arrow[r, "{(1_X,0)}"] \arrow[rd, "k"'] & X \times X \arrow[d, "\varphi", dotted] & X \arrow[l, "\diag_X"'] \arrow[ld, "e \comp \del"] \\
            & X \rtimes_\xi B
        \end{tikzcd}
    \]
    where the dotted arrow \(\varphi\) is given by the Peiffer Condition \(\Commut{k}{e \comp \del}_{\conjact{X}{X}} = 0\). By \cref{prop:precomp commutativity}, this means that \(0 = (e \comp \del) \comp \kappa\) and \(k\) commute relatively to the cospan \begin{tikzcd}[cramped]
        K \arrow[r, "{(\kappa,1_K)}"] & X \times K & X \arrow[l, "{(1_X,0)}"']
    \end{tikzcd} so, by \cref{prop:postcomp commutativity}, it is also the case for \(0\) and \(1_X\) since \(k\) is monic, which lets us conclude that \(\kappa\) is central by \cite[Theorem~4.1]{Sha23} (see \cref{exs:commutativity}).

    Moreover, let us remark that, for these two proofs of \cref{lemma:kappa central}, we only needed the Peiffer Condition and not the Precrossed Module Condition.
\end{remark}

Now, we consider a central extension \(\del \from X \to B\) and will see how it naturally has a (unique) structure of crossed module. By \cref{rem:Peiffer graph}, the action \(\xi\) of this crossed module \((\del,\xi)\) must make the square
\begin{eqdiagr}
    \label{diagr:def xi del reg epi}
    \begin{tikzcd}[squared]
        X \flat X \arrow[r, "\conjact{X}{X}"] \arrow[d, "\del \flat 1_X"'] & X \arrow[d, equal] \\
        B \flat X \arrow[r, "\xi"'] & X
    \end{tikzcd}
\end{eqdiagr}
commute. Using the centrality of \(\del\) and the fact that it is a regular epimorphism, one can show that such an arrow \(\xi \from B \flat X \to X\) exists~\cite[Example~6.2]{HVdL11}. (For example, in a variety of algebras, \(\del\) is a quotient and this diagram tells us that the action by an element \(b \in B\) is given by choosing a representative \(x\) of \(b\) (i.e.\ an element \(x \in X\) such that \(\del(x) = b\)) and acting by conjugation with this \(x\), the centrality of \(\del\) assuring that this procedure is well-defined.) Once we know the existence of this arrow, we can show that it is indeed an action using the fact that \(\conjact{X}{X}\) is an action and that \(\del \flat 1_X\) is (regularly) epic~\cite[Lemma~2.6]{dMVdL20}. Let us now check the two commutator conditions. By \cref{rem:Peiffer graph}, the square we used to define \(\xi\) is equivalent to the condition \PFF{}\@. For \PCM{}\@, consider the following diagram.
\[
    \begin{tikzcd}[squared]
        X \flat X \arrow[r, "\conjact{X}{X}"] \arrow[d, "\del \flat 1_X"'] & X \arrow[d, equal] \\
        B \flat X \arrow[r, "\xi"] \arrow[d, "B \flat \del"'] & X \arrow[d, "\del"] \\
        B \flat B \arrow[r, "\conjact{B}{B}"'] & B
    \end{tikzcd}
\]
By \cref{lemma:equivariance => commutativity}, we want to show that the bottom square is commutative. Since the top square is commutative by construction of \(\xi\) and the arrow \(\del \flat 1_X\) is epic, we must just check that the outer rectangle is also commutative, which is indeed the case since
\[
    \begin{split}
        \conjact{B}{B} \comp \bigl((B \flat \del) \comp (\del \flat 1_X)\bigr) &= (\CoindArr{1_B,1_B} \comp \kerb{B}{B}) \comp (\del \flat \del) = \CoindArr{1_B,1_B} \comp \bigl((\del+\del) \comp \kerb{X}{X}\bigr) \\
        &= \CoindArr{\del,\del} \comp \kerb{X}{X} = \del \comp (\CoindArr{1_X,1_X} \comp \kerb{X}{X}) = \del \comp \conjact{X}{X}\text{.}
    \end{split}
\]

\subsubsection*{Conclusion}
To summarise, when \(\del \from X \to B\) is a monomorphism, the Precrossed Module Condition corresponds exactly to the existence of a conjugation action of \(B\) on~\(X\), which is a characterisation of normal subobjects in a semi-abelian context. Moreover, when it is the case, the Peiffer Condition follows automatically. For a regular epimorphism \(\del \from X \to B\), the condition \PFF{} implies that \(\del\) is central (even in a more general case without any assumption on \(\del\)) so we have a central extension. Conversely, given a central extension \(\del \from X \to B\), the condition \PFF{} means that the action \(\xi\) makes \cref{diagr:def xi del reg epi} commute and, using the centrality of \(\del\), we may show that such an action exists. Then, once we have this action \(\xi\) making \cref{diagr:def xi del reg epi} commute, the condition \PCM{} comes for free.

\section{Further work}
The theory of twisted commutators can be further developed in several different ways. We name a few.

First of all, more examples of twisted commutators can be found ``in nature''; we are expecting that, often, the existence of an arrow may be seen as twisted commutativity. One situation which we would like to better understand from our perspective is the so-called \emph{Peiffer commutator} of~\cite{CMM17}. In fact, there are several instances of twisted commutativity in that article.

Another way to extend the current theory would be beyond the condition \SH{}\@. In the case of the classical commutator, this involves ternary commutators, as explained in~\cite{HVdL13}. So, is there an appropriate concept of \emph{higher-order twisted cosmash product}? And a corresponding \emph{higher-order twisted commutator}? In the classical case, there is a join decomposition formula for subobjects \(K\), \(L\), \(M \leq X\):
\[
    [K,L \join M] = [K,L] \join [K,M] \join [K,L,M]\text{,}
\]
which explains the sudden appearance of a ternary commutator. Does a similar decomposition make sense for twisted commutators? Can we perhaps decompose a twisted commutator in terms of ordinary commutators, or decompose ordinary commutators in terms of the twisted commutator? The same question arises for more exotic types of commutators, such as, in particular, the ones considered in~\cite{SVdL21}---which is what initially motivated us to start these investigations.

Finally, we believe it would be interesting to explore relations with categorical Galois theory~\cite{BJ01}, with the aim of extending the use of Galois theory to the study of abelian non-central extensions.

\section*{Acknowledgements}
Many thanks to Nelson Martins-Ferreira for fruitful discussions on the subject of the article and, in particular, for suggesting the level of generality in \cref{def:rel commutativity}.

\printbibliography

\end{document}